\documentclass{article}
\usepackage{amssymb,amsmath}
\newtheorem{theorem}{Theorem}[section]{}
{}
{}
\newtheorem{example}{Example}[section]

\def\hs{\hskip0.04cm'}
\def\dz{,\kern-0.1em ,}
\def\dd#1{{#1\kern-0.4em\char"16\kern-0.1em}}
\def\D#1{{\raise0.2ex\hbox{-}\kern-0.42em #1}}
\def\Bbb{\mathbb}

\def\la{\lambda}

\def\t{\theta}
\def\a{\alpha}
\def\b{\beta}
\def\G{{\rm\Gamma}}
\def\g{\gamma}

\def\z{\zeta}

\def\Re{{\rm Re\,}}

\def\rmd{\mathrm{d}}

\numberwithin{equation}{section}

\title{Series over Bessel functions as series in terms of Riemann's zeta function}

\author{Slobodan B. Tri\v ckovi\'c$^1$, Miomir S. Stankovi\'c$^2$\\[2mm]
$^1$\small University of Ni\v s, Department of Mathematics,\\[-1mm]
\small ORCID: 0000-0002-3766-8579, e-mail: sbt@junis.ni.ac.rs\\[-1mm]
$^2$\small Mathematical Institute of the Serbian Academy of\\[-1mm]
\small Sciences and Arts, Belgrade, Serbia}

\date{\today}

\begin{document}

\maketitle

\begin{abstract} Relying on the Hurwitz formula, we find sums of the series over sine and
cosine functions through the Hurwitz zeta function. Using another summation formula for these
trigonometric series, we find finite sums of some series over the Riemann zeta function.
\end{abstract}



\section{Main results}

Apostol \cite{apostol} derived Hurwitz's formula
\begin{equation*}
\z(1-s,a)=\frac{\G(s)}{(2\pi)^s}\Big(e^{-i\pi s/2}\sum_{n=1}^\infty\frac{e^{2i\pi na}}{n^s}
+e^{i\pi s/2}\sum_{n=1}^\infty\frac{e^{-2i\pi na}}{n^s}\Big),
\end{equation*}
where $\z(1-s,a)$ is the Hurwitz zeta function with $0<a\leqslant 1$. We can write it in the equivalent form
\begin{equation}\label{hurwitz-formula}
\z(1-s,a)=\frac{2\,\G(s)}{(2\pi)^s}\sum_{n=1}^\infty\frac{1}{n^s}\cos(\tfrac{\pi}2s-2n\pi a),
\end{equation}
which enables us to find the limiting value
\begin{equation*}
\lim_{s\to2m-1}\big(\z(1-s,a)+\z(1-s,1-a)\big)=\lim_{s\to2m-1}\Big(\frac{4\,\G(s)}{(2\pi)^s}
\sum_{n=1}^\infty\frac{\cos\frac{\pi s}2\cos2n\pi a}{n^s}\Big)=0,
\end{equation*}
where $m\in\mathbb N$.

\begin{theorem} For $\Re s>1$ and $0<a\leqslant 1$ there hold the following formulas
\begin{equation}\label{sin-cos}
\begin{split}
&\sum_{n=1}^\infty\frac{\sin(nx+y)}{n^s}=\frac{(2\pi)^s}{2\,\G(s)\sin\pi s}
\Big(\cos\Big(y-\frac{\pi s}2\Big)\z\Big(1-s,\frac x{2\pi}\Big)\\
&\hskip4cm-\cos\Big(y+\frac{\pi s}2\Big)\z\Big(1-s,1-\frac x{2\pi}\Big)\Big),\\
&\sum_{n=1}^\infty\frac{\cos(nx+y)}{n^s}=\frac{(2\pi)^s}{2\,\G(s)\sin\pi s}
\Big(\sin\Big(y+\frac{\pi s}2\Big)\z\Big(1-s,1-\frac x{2\pi}\Big)\\
&\hskip4cm-\sin\Big(y-\frac{\pi s}2\Big)\z\Big(1-s,\frac x{2\pi}\Big)\Big),
\end{split}
\end{equation}
where $y\in\mathbb R$.
\end{theorem}

\begin{proof} We replace $a$ first with $1-x/2\pi$, then with $x/2\pi$. Further,
we multiply the first equality by $e^{i(y+\pi s/2)}$, and the latter by $e^{i(y-\pi s/2)}$.
After subtracting, we obtain
\begin{align*}
&\hskip-1.5cm e^{i(y+\pi s/2)}\z\Big(1-s,1-\frac x{2\pi}\Big)-e^{i(y-\pi s/2)}\z\Big(1-s,\frac x{2\pi}\Big)\\
&=\frac{\G(s)}{(2\pi)^s}\big(e^{i\pi s}-e^{-i\pi s}\big)\sum_{n=1}^\infty\frac{e^{i(nx+y)}}{n^s}\\
&=\frac{2\,\G(s)\sin\pi s}{(2\pi)^s}\Big(i\sum_{n=1}^\infty\frac{\cos(nx+y)}{n^s}
-\sum_{n=1}^\infty\frac{\sin(nx+y)}{n^s}\Big).
\end{align*}
Hence, comparing the corresponding real and imaginary parts, we arrive at \eqref{sin-cos}. \end{proof}

For $y=0$, we obtain
\begin{equation*}
\sum_{n=1}^\infty\frac{\cos nx}{n^s}=\frac{(2\pi)^s}{4\,\G(s)}
\frac{\z\big(1-s,1-\frac x{2\pi}\big)+\z\big(1-s,\frac x{2\pi}\big)}{\cos\frac{\pi s}2}.
\end{equation*}
Thus, we evaluate
\begin{align}
\sum_{n=1}^\infty\frac{\cos nx}{n^{2m-1}}&=\frac{(2\pi)^{2m-1}}{4(2m-2)!}
\lim_{s\to2m-1}\frac{\z\big(1-s,1-\frac x{2\pi}\big)+\z\big(1-s,\frac x{2\pi}\big)}{\cos\frac{\pi s}2}\label{cos-hurwitz}\\
&=\frac{(-1)^{m-1}(2\pi)^{2m-2}}{(2m-2)!}\Big(\z\hs\Big(2-2m,1-\frac x{2\pi}\Big)+\z\hs\Big(2-2m,\frac x{2\pi}\Big)\Big).\notag
\end{align}
We shall use \eqref{cos-hurwitz} to obtain a finite sum of the series over the zeta functions.
\begin{theorem} For $m\in\mathbb N$, there holds
\begin{align}
&\hskip-0.35cm\sum_{k=1}^\infty\frac{\z(2k)}{(2k)_{2m-1}}\Big(\frac x{2\pi}\Big)^{2k}
=\frac{\log x-H_{2m-2}}{2(2m-2)!}
+\frac{(\frac{2\pi}x)^{2m-2}}{2(2m-2)!}\bigg(\z\hs\Big(2-2m,1-\frac x{2\pi}\Big)\notag\\
&+\z\hs\Big(2-2m,\frac x{2\pi}\Big)\bigg)
+(-1)^{m}\sum_{k=0}^{m-2}\frac{(-1)^kx^{2k-2m+2}\z(2m-2k-1)}{2(2k)!}.\label{zeta-pochh}
\end{align}
\end{theorem}

\begin{proof} Taking account of \cite{intr-19-6}
\begin{equation*}
\sum_{n=1}^\infty\frac{\cos nx}{n^s}=\frac{\pi x^{s-1}}{2\G(s)\cos\frac{\pi s}2}
+\sum_{k=0}^\infty\frac{(-1)^k\z(s-2k)}{(2k)!}\,x^{2k},
\end{equation*}
after letting $s\to2m-1$, it follows
\begin{align}
&\sum_{n=1}^\infty\frac{\cos nx}{n^{2m-1}}=\sum_{k=0}^{m-2}\frac{(-1)^k\z(2m-2k-2)}{(2k)!}\,x^{2k}
+\lim_{s\to2m-1}\Big(\frac{\pi x^{s-1}}{2\G(s)\cos\frac{\pi s}2}\label{remainder}\\
&\hskip1cm+\frac{(-1)^{m-1}\z(s-2m+2)}{(2m-2)!}\,x^{2m-2}\Big)+\sum_{k=m}^\infty\frac{(-1)^k\z(2m-2k-1)}{(2k)!}\,x^{2k}.\notag
\end{align}
By bringing the fractions in brackets to the same denominator and ap\-ply\-ing L'H\^{o}pital's rule,
we determine the above limiting value
\begin{multline*}
\lim_{s\to2m-1}\Big(\frac{\pi x^{s-1}}{2\G(s)\cos\frac{\pi s}2}
+\frac{(-1)^{m-1}\z(s-2m+2)}{(2m-2)!}\,x^{2m-2}\Big)\\
=\frac{(-1)^{m}x^{2m-2}}{(2m-2)!}\big(\log x-H_{2m-2}\big).
\end{multline*}
As for the remainder in \eqref{remainder}, we introduce a new running index by $k=m+j-1$,
then reverting to $k$ and recalling the relation
\begin{equation}\label{z(1-2n)}
\z(1-2n)=(-1)^n\frac{2(2n-1)!\z(2n)}{(2\pi)^{2n}}=(-1)^n\frac{2\,\G(2n)\z(2n)}{(2\pi)^{2n}},
\end{equation}
we come to the formula
\begin{align*}
&\hskip-0.2cm\sum_{n=1}^\infty\frac{\cos nx}{n^{2m-1}}
=\frac{(-1)^{m}x^{2m-2}}{(2m-2)!}\big(\log x-H_{2m-2}\big)
+\sum_{k=0}^{m-2}\frac{(-1)^kx^{2k}\z(2m-2k-1)}{(2k)!}\\
&\hskip3cm +2(-1)^{m-1}x^{2m-2}\sum_{k=1}^\infty\frac{\z(2k)}{(2k)_{2m-1}}\Big(\frac x{2\pi}\Big)^{2k}.
\end{align*}
Using \eqref{cos-hurwitz}, we obtain \eqref{zeta-pochh}. \end{proof}

We are dealing now with a more general series.

\begin{theorem} For $m,p\in\mathbb N$ and $0<x<2\pi$, there holds
\begin{equation}\label{zeta-pochh-mp}
\begin{split}
&\hskip-0.5cm\sum_{n=1}^\infty\frac{(m+n)_p}{(2n)_{2p+2m}}\z(2n)\Big(\frac x{2\pi}\Big)^{2n}\\
&=(-1)^{m-1}4x^{p+2m-2}\sum_{k=1}^p C_k
\sum_{j=0}^{2k-1}(-1)^j\prod_{i=1}^j(2k-i)\sum_{n=1}^\infty\frac{\z(2n)(\frac x{2\pi})^{2n}}{(2n)_{2m+j+1}}.
\end{split}
\end{equation}
\end{theorem}

\begin{proof} We consider a decomposition of the fraction of two Pochhammer symbols in \eqref{zeta-pochh-mp}
\begin{align*}
\frac{(m+n)_p}{(2n)_{2p+2m}}&=\frac{(m+n)\cdots(m+n+p-1)}{(2n)_{2m}(2n+2m)\cdots(2n+2m+2p-1)}\\
&=\frac1{2^p(2n)_{2m}(2n+2m+1)(2n+2m+3)\cdots(2n+2m+2p-1)}.
\end{align*}
By applying Heaviside's method, in the next step, we make a decomposition of the rational function
\begin{equation*}
\frac{P(n)}{Q(n)}=\frac1{(2n+2m+1)(2n+2m+3)\cdots(2n+2m+2p-1)},
\end{equation*}
with $P(n)=1,Q(n)=(2n+2m+1)(2n+2m+3)\cdots(2n+2m+2p-1)$. Thus, we have
\begin{equation*}
\frac{P(n)}{Q(n)}=\frac{C_1}{n-a_1}+\frac{C_2}{n-a_2}+\cdots+\frac{C_p}{n-a_p},
\end{equation*}
where the constants $C_k,\,k=1,\dots,p$, are to be determined, and $Q(a_k)=0$.
So, we have
$$
C_k=\frac{P(a_k)}{Q\hs(a_k)},\quad a_k=-(2m+2k-1)/2,\quad k=1,\dots,p,
$$
which means that the left-hand side series in \eqref{zeta-pochh-mp} we can express as follows
\begin{equation}\label{ck}
\sum_{n=1}^\infty\frac{(m+n)_p}{(2n)_{2p+2m}}\z(2n)\Big(\frac x{2\pi}\Big)^{2n}
=\sum_{k=1}^p\frac{C_k}{2^{p-1}}\sum_{n=1}^\infty\frac{\z(2n)\big(\frac x{2\pi}\big)^{2n}}{(2n)_{2m}(2n+2m+2k-1)}.
\end{equation}
The whole idea consists in finding constants $c_1,\dots,c_{2k}$ in the representation
\begin{equation*}
\frac1{2n+2m+2k-1}=\frac{c_1}{2n+2m}+\frac{c_2}{(2n+2m)_2}+\cdots+\frac{c_{2k}}{(2n+2m)_{2k}},
\end{equation*}
where $k=1,\dots,p$. We multiply the numerator and denominator of the left-hand side fraction
by the missing factors between $(2n+2m+2k-1)$ and $(2n+2m)$, including the latter. Also, we bring
the right-hand side fractions to the same denominator. Thus, we obtain the equality of the numerators
\begin{multline*}
(2n+2m)\cdots(2n+2m+2k-2)=c_1(2n+2m+1)\cdots(2n+2m+2k-1)\\
+c_2(2n+2m+2)\cdots(2n+2m+2k-1)+\cdots+c_{2k},\quad k=1,\dots,p.
\end{multline*}
The highest power on the left-hand side and the product at $c_1$ is $(2n)^{2k-1}$. There follows $c_1=1$.
After rearrangements, we have
\begin{align}
&-(2k-1)(2n+2m+1)\cdots(2n+2m+2k-2)\notag\\
&=c_2(2n+2m+2)\cdots(2n+2m+2k-1)+c_3(2n+2m+3)\cdots\times\label{c2}\\
&\cdots(2n+2m+2k-1)+\cdots+c_{2k-1}(2n+2m+2k-1)+c_{2k}.\notag
\end{align}
Regarding both sides as polynomial functions in $n$ of the degree $2k-2$,
we take the $(2k-2)$th derivative. As a result, we have
\begin{equation*}
-(2k-1)2^{2k-2}(2k-2)!=c_22^{2k-2}(2k-2)!.
\end{equation*}
Hence, we find $c_2=-(2k-1)$. Replacing this value in \eqref{c2}, we can determine $c_3$ similarly
\begin{align*}
&(2k-1)(2k-2)(2n+2m+2)\cdots(2n+2m+2k-2)\notag\\
&\hskip2cm=c_3(2n+2m+3)\cdots(2n+2m+2k-1)+\\
&c_4(2n+2m+4)\cdots(2n+2m+2k-1)+\cdots+c_{2k-1}(2n+2m+2k-1)+c_{2k}.\notag
\end{align*}
Now, the polynomials on both sides are of the degree $2k-3$, and we take the $(2k-3)$th derivative,
and obtain
\begin{equation*}
(2k-1)(2k-2)2^{2k-3}(2k-3)!=c_32^{2k-3}(2k-3)!,
\end{equation*}
which implies $c_3=(2k-1)(2k-2)$. By repeating this procedure, we arrive at the relation
\begin{equation*}
\frac1{(2n)_{2m}(2n+2m+2k-1)}=
\sum_{j=0}^{2k-1}\frac{(-1)^j}{(2n)_{2m+j+1}}\prod_{i=1}^j(2k-i).
\end{equation*}
Together with \eqref{ck}, this gives rise to \eqref{zeta-pochh-mp}. \end{proof}

Thus, the procedure to evaluate the left-hand side sum of \eqref{zeta-pochh-mp} reduces to
a multiple application of \eqref{zeta-pochh}.

\section{Applications to some series over Bessel\\ functions}

Bessel functions\index{function!Bessel's} $J_\nu(x)$ defined by the Swiss mathematician Daniel Bernoulli,
then generalized and developed by Friedrich Bessel while studying the dynamics of gravitational systems
in the second decade of the 19th century, are canonical solutions of homogeneous Bessel's differential
equation
\begin{equation}\label{bess-diff-eq}
x^2\frac{\rmd^2y}{\rmd x^2}+x\frac{\rmd y}{\rmd x}+(x^2-\nu^2)y=0
\end{equation}
for an arbitrary complex number $\nu$. The particular solution
\begin{equation}\label{bess-sum}
J_\nu(x)=\sum_{m=0}^\infty
\frac{(-1)^m(\frac x2)^{2m+\nu}}{m!\G(m+\nu+1)},
\end{equation}
is called Bessel's function of the first kind order $\nu$ \cite{pruds}.
Apart from that, it can be represented by Poisson's integral
\begin{equation}\label{poisson}
J_\nu(z)=\frac{2\left(\frac z2\right)^\nu} {\G\left(\frac12\right)\G\left(\nu+\frac 12\right)}
\int_0^{\pi/2}\sin^{2\nu}\t\,\cos(z\cos\t)\,d\t,\quad\Re\nu>-\textstyle\frac12,
\end{equation}
proved by Poisson \cite{poisson} and Lommel \cite{lommel} to be a solution of homogenous Bessel's
differential equation for $2\nu\in\Bbb N_0$, relying on the summation of trigonometric series,
in \cite{zeit-2006}, we derived summation formulas for the series
\begin{equation}\label{10*}
\sum_{n=1}^\infty\frac{J_\nu(nx)}{n^\a}=\frac{\pi(\frac x2)^{\a-1}\sec\frac{\pi(\a-\nu)}2}
{2\hskip0.4mm\G(\frac{\a-\nu+1}2)\G(\frac{\a+\nu+1}2)}
+\sum_{k=0}^\infty\frac{(-1)^k(\frac x2)^{\nu+2k}\z(\a-\nu-2k)}{k!\,\G(\nu+k+1)},
\end{equation}
where $\a>0,\,\a>\nu>-\frac12,\,0<x<2\pi$.

Providing $\a-\nu=2m$, $m\in\Bbb N$, the right-hand side series in \eqref{10*} truncates because
Riemann's zeta function equals zero if the argument is a negative even integer. So, setting $\a=\nu+2m$
in \eqref{10*} brings the series in closed form
\begin{equation}\label{10*closed}
\sum_{n=1}^\infty\frac{J_\nu(nx)}{n^{\nu+2m}}
=\frac{m!(-1)^m\,x^{\nu+2m-1}\sqrt\pi}{2^\nu(2m)!\G(\nu+m+\frac12)}
+\sum_{k=0}^m\frac{(-1)^k\,\z(2m-2k)(\frac x2)^{\nu+2k}}{k!\,\G(\nu+k+1)},
\end{equation}
where $0<x<2\pi,\,\Re\nu>-2m-\frac12$.

In the case $\a-\nu=2m-1,\,m\in\Bbb N$, we cannot immediately place this in the right-hand side of
the relation \eqref{10*} because one encounter singularities, i.e. in $\sec\frac{\pi(\a-\nu)}2$
within the numerator of the first term and in the member of the right-hand series for the index
$k=m-1$.

\begin{theorem} For $\a=\nu+2m-1,\,m\in\Bbb N$, there holds
\begin{multline}\label{1/2}
\sum_{n=1}^\infty\frac{J_\nu(nx)}{n^{\nu+2m-1}}=\frac{(-1)^{m-1}(\frac x2)^{\nu+2m-2}}{2\G(m)\G(\nu+m)}
\Big(H_{m-1}+H_{\nu+m-1}-2\ln\frac x2\,\Big)\\
+\sum_{k=0}^{m-2}\frac{(-1)^k\z(2m-2k-1)(\frac x2)^{\nu+2k}}{k!\G(\nu+k+1)}
+\sum_{k=1}^\infty\frac{\G(2k)\z(2k)(\frac x{4\pi})^{2k}}{\G(m+k)\G(\nu+m+k)}.
\end{multline}
\end{theorem}

\begin{proof}Because of what is said above, it is necessary to take the limiting value in \eqref{10*} when
$\a\to\nu+2m-1$ 
\begin{equation*}
\lim_{\a\to \nu+2m-1}\Big(\frac{\frac\pi2(\frac x2)^{\a-1}\sec\frac{\pi(\a-\nu)}2}
{\G(\frac{\a-\nu+1}2)\G(\frac{\a+\nu+1}2)}
+\sum_{k=0}^{m-1}\frac{(-1)^k\z(\a-\nu-2k)(\frac x2)^{\nu+2k}}{k!\G(\nu+k+1)}.
\end{equation*}
For $k=0,1,\dots,m-2$ all the terms have no singularities if $\a=\nu+2m-1$,
so it suffices to deal only with the term for $k=m-1$, then we find
\begin{multline*}
\lim_{\a\to \nu+2m-1}\Big(\frac{\frac\pi2(\frac x2)^{\a-1}\sec\frac{\pi(\a-\nu)}2}
{\G(\frac{\a-\nu+1}2)\G(\frac{\a+\nu+1}2)}+\frac{(-1)^{m-1}(\frac x2)^{\nu+2m-2}\z(\a-\nu-2m+2)}{(m-1)!\,\G(\nu+m)}\Big)\\
=\frac{(-1)^{m-1}(\frac x2)^{\nu+2m-2}}{2\G(m)\G(\nu+m)}
\Big(H_{m-1}+H_{\nu+m-1}-2\log\frac x2\,\Big),
\end{multline*}
where $\g$ is Euler-Mascheroni's constant. Thus, we have
\begin{multline*}
\sum_{n=1}^\infty\frac{J_\nu(nx)}{n^{\nu+2m-1}}=\frac{(-1)^{m-1}(\frac x2)^{\nu+2m-2}}{2\G(m)\G(\nu+m)}
\Big(H_{m-1}+H_{\nu+m-1}-2\log\frac x2\,\Big)\\
+\sum_{k=0}^{m-2}\frac{(-1)^k\z(2m-2k-1)(\frac x2)^{\nu+2k}}{k!\G(\nu+k+1)}
+\sum_{k=m}^\infty\frac{(-1)^k\z(2m-2k-1)(\frac x2)^{\nu+2k}}{k!\,\G(\nu+k+1)}.
\end{multline*}

As for the remainder, we transform it by shifting the index, so we introduce the substitution $k=m+j-1$,
then revert to $k$ instead of using $j$, i.e.
\begin{multline*}
\sum_{k=m}^\infty\frac{(-1)^k\z(2m-2k-1)(\frac x2)^{\nu+2k}}{k!\,\G(\nu+k+1)}\\
=(-1)^{m-1}(\tfrac x2)^{\nu+2m-2}\sum_{k=1}^\infty\frac{(-1)^k\z(1-2k)(\frac x2)^{2k}}{\G(m+k)\G(\nu+m+k)}.
\end{multline*}
Using \eqref{z(1-2n)} substituting there $k$ for $n$, the last series becomes
\begin{equation*}
2(-1)^{m-1}(\tfrac x2)^{\nu+2m-2}\sum_{k=1}^\infty\frac{\G(2k)\z(2k)(\frac x{4\pi})^{2k}}{\G(m+k)\G(\nu+m+k)},
\end{equation*}
whereby we arrive at \eqref{1/2}. \end{proof}

In a special case, setting $\nu=\frac12$ in \eqref{1/2}, we have
\begin{align}
&\hskip-0.4cm\sum_{n=1}^\infty\frac{J_{1/2}(nx)}{n^{2m-1/2}}=\frac{(-1)^{m-1}(\frac x2)^{2m-3/2}}{2\G(m)\G(\frac12+m)}
\Big(H_{m-1}+H_{m-1/2}-2\log\frac x2\,\Big)\label{nu-2m-1}\\
&+\sqrt{\frac x2}\sum_{k=0}^{m-2}\frac{(-1)^k\z(2m-2k-1)(\frac x2)^{2k}}{k!\G(\frac12+k+1)}
+\sum_{k=1}^\infty\frac{\G(2k)\z(2k)(\frac x{4\pi})^{2k}}{\G(m+k)\G(\frac12+m+k)}.\notag
\end{align}

After some modifications and referring to \eqref{zeta-pochh}, we bring the right-hand series in \eqref{nu-2m-1}
 to the closed form, i.e.
\begin{multline}\label{z-2k-fin}
\sum_{k=1}^\infty\frac{\G(2k)\z(2k)(\frac x{4\pi})^{2k}}{\G(m+k)\G(\frac12+m+k)}
=\frac{(-1)^{m-1}}{\sqrt{2\pi}}8x^{2m-3/2}\sum_{k=1}^\infty\frac{\z(2k)}{(2k)_{2m}}\Big(\frac x{2\pi}\Big)^{2k}\\
=\frac{(-1)^m4x^{2m-3/2}}{(2m-1)!\sqrt{2\pi}}(H_{2m-1}-\log2\pi)
-\frac{4}{\sqrt{2x\pi}}\sum_{k=0}^{m-2}\frac{(-1)^k\z(2m-2k-1)x^{2k+1}}{(2k+1)!}\\
+\frac{(-1)^m4(2\pi)^{2m-3/2}}{(2m-1)!\sqrt{x}}\bigg(\z\hs\Big(1-2m,1-\frac x{2\pi}\Big)
-\z\hs\Big(1-2m,1+\frac x{2\pi}\Big)\bigg).
\end{multline}
By differentiating the basic property of Hurwitz's function
\begin{equation*}
\z(s,a)=a^{-s}+\z(s,a+1)
\end{equation*}
with respect to $s$ and putting then $a=\frac x{2\pi}$ and $s=1-2m$, for the last term of \eqref{z-2k-fin},
we find
$$
-\z\hs\Big(1-2m,1+\frac x{2\pi}\Big)=-\Big(\frac x{2\pi}\Big)^{2m-1}\log\frac x{2\pi}-\z\hs\Big(1-2m,\frac x{2\pi}\Big).
$$
We can rewrite now \eqref{z-2k-fin} and use it to give the left-hand side series of \eqref{nu-2m-1}
the closed form
\begin{align}
&\hskip-0.05cm\sum_{n=1}^\infty\frac{J_{1/2}(nx)}{n^{2m-1/2}}
=\frac{(-1)^m4(2\pi)^{2m-1}}{(2m-1)!\sqrt{2\pi x}}
\bigg(\frac12\big(H_{2m-1}-\log x\,\big)\Big(\frac x{2\pi}\Big)^{2m-1}\label{gen}\\
 &+\z\hs\Big(1-2m,1-\frac x{2\pi}\Big)\!-\!\z\hs\Big(1-2m,\frac x{2\pi}\Big)\bigg)\!
+\!\sqrt{\frac{2x}\pi}\sum_{k=0}^{m-2}\frac{(-x^2)^k\z(2m-2k-1)}{(2k+1)!}.\notag
\end{align}


\subsection{ Series over spherical Bessel functions}

We can generalize \eqref{gen} by considering spherical Bessel functions 
$j_k(z)$, linked to Bessel functions of the first kind by the equality
$$
j_p(x)=\sqrt{\frac\pi{2x}}\,J_{p+\frac12}(x),\quad p\in\mathbb N_0.
$$
So, multiplying both sides of \eqref{10*}
by $\sqrt{\dfrac\pi{2x}}$ and setting $\nu=p+\frac12,\,p\in\mathbb N_0$, we have
\begin{multline}\label{bessel-spher}
\hskip-0.2cm\sum_{n=1}^\infty\frac{j_p(nx)}{n^\a}
=\frac{\pi\sqrt\pi(\frac x2)^{\a-\frac32}}{2\,\G(\frac{\a-p}2+\frac14)\G(\frac{\a+p}2+\frac34)\cos(\frac{\pi(\a-p)}2-\frac\pi4)}\\
+\frac{\sqrt\pi}2\sum_{k=0}^\infty\frac{(-1)^k\z(\a-p-\frac12-2k)(\frac x2)^{p+2k}}{k!\,\G(p+\frac12+k+1)}.
\end{multline}
Here, we are referring to Euler's reflection formula
\begin{equation*}
\G(1-s)\G(s)=\frac\pi{\sin\pi s},\quad s\notin\mathbb Z,
\end{equation*}
with $s=\frac{\a-p}2+\frac14$, and obtain
\begin{equation*}
\G\big(\tfrac34-\tfrac{\a-p}2\big)\G\big(\tfrac{\a-p}2+\tfrac14\big)
=\frac\pi{\cos(\frac\pi4-\frac{\pi(\a-p)}2)}=\frac\pi{\cos(\frac{\pi(\a-p)}2-\frac\pi4)}.
\end{equation*}
Replacing this
in the first term of \eqref{bessel-spher}, and using the property
\begin{equation*}
\G\Big(n+\frac12\Big)=\frac{\sqrt\pi\,\G(2n)}{2^{2n-1}\G(n)}=\frac{(2n)!}{2^{2n}n!}\sqrt\pi,
\end{equation*}
we obtain the summation formula for the series over spherical Bessel functions 
\begin{multline*}
\sum_{n=1}^\infty\frac{(s)^{n-1}j_p((an-b)x)}{(an-b)^\a}
=\frac{c\sqrt\pi x^{\a-\frac32}\G(\frac34-\frac{\a-p}2)}{2^{\a+\frac12}\G(\frac34+\frac{\a+p}2)}\\
+(2x)^p\sum_{k=0}^\infty\frac{(-1)^k\G(p+k+1)F(\a-p-\frac12-2k)x^{2k}}{k!\,\G(2p+2k+2)}.
\end{multline*}
Now, for $a=1,b=0,s=1,c=1,F=\z$, we first evaluate
\begin{multline*}
\lim_{\a\to p+2m-1/2}\Big(\frac{\sqrt\pi x^{\a-\frac32}\G(\frac34-\frac{\a-p}2)}{2^{\a+\frac12}\G(\frac34+\frac{\a+p}2)}\\
\hskip3cm+2^{p}\frac{(-1)^{m-1}\G(p+m)\z(\a-p-2m+\frac32)x^{2m+p-2}}{(m-1)!\G(2p+2m)}.
\end{multline*}
By bringing to the same denominator and applying the L'H\^opital rule, after the limiting process, we obtain
\begin{multline*}
\sum_{n=1}^\infty\frac{j_p(nx)}{n^{p+2m-1/2}}
=\frac{(-1)^{m-1}(\frac x2)^{p+2m-\frac32}\sqrt{\frac\pi{8x}}}{\G(m)\G(p+\frac12+m)}
\Big(H_{m-1}+H_{p+m-\frac12}-2\log\frac x2\,\Big)\\
+(2x)^p\sum_{k=0}^{m-2}\frac{(-1)^k\G(p+k+1)\z(2m-1-2k)x^{2k}}{k!\,\G(2p+2k+2)}\\
+(-1)^{m-1}2^{p+1}x^{p+2m-2}\sum_{n=1}^\infty\frac{(m+n)_p}{(2n)_{2p+2m}}\z(2n)\Big(\frac x{2\pi}\Big)^{2n}.
\end{multline*}
Referring to \eqref{zeta-pochh-mp}, we evaluate 
\begin{multline}\label{bess-sph}
\sum_{n=1}^\infty\frac{j_p(nx)}{n^{p+2m-1/2}}
=\frac{(-1)^{m-1}(\frac x2)^{p+2m-\frac32}\sqrt{\frac\pi{8x}}}{\G(m)\G(p+\frac12+m)}
\Big(H_{m-1}+H_{p+m-\frac12}-2\log\frac x2\,\Big)\\
+(2x)^p\sum_{k=0}^{m-2}\frac{(-1)^k\G(p+k+1)\z(2m-1-2k)x^{2k}}{k!\,\G(2p+2k+2)}\\
+(-1)^{m-1}4x^{p+2m-2}\sum_{k=1}^p C_k\sum_{j=0}^{2k-1}(-1)^j\prod_{i=1}^j(2k-i)
\sum_{n=1}^\infty\frac{\z(2n)(\frac x{2\pi})^{2n}}{(2n)_{2m+j+1}}.
\end{multline}
We apply \eqref{zeta-pochh} to evaluate the sum of the series on the right side of \eqref{bess-sph}.

\begin{example}\sl 
For $m=3$ and $p=2$, we apply \eqref{bess-sph} and \eqref{zeta-pochh} obtain
\begin{multline*}
\sum_{n=1}^\infty\frac{j_2(nx)}{n^{15/2}}=\frac{x^6\left(\frac{3197}{1260}-\log x\right)}{15120}
-\frac{x^4 \zeta (3)}{420}+\frac{x^2 \zeta (5)}{30}+\frac{4 \pi ^7\zeta\hs\left(-7,\frac{x}{2 \pi }\right)}{315 x}\\
-\frac{4 \pi ^7 \zeta\hs\left(-7,1-\frac{x}{2\pi }\right)}{315 x}-\frac{\pi ^8 \zeta\hs\left(-8,\frac{x}{2 \pi }\right)}{105 x^2}
-\frac{\pi ^8 \zeta\hs\left(-8,1-\frac{x}{2 \pi }\right)}{105 x^2}\\
+\frac{2 \pi ^9 \zeta\hs\left(-9,\frac{x}{2 \pi
   }\right)}{945 x^3}-\frac{2 \pi ^9 \zeta\hs\left(-9,1-\frac{x}{2 \pi }\right)}{945 x^3}.
   \end{multline*}
\end{example}

\end{document}